\documentclass[letter,10pt]{article}
\usepackage{amssymb}
\usepackage{amsthm}
\usepackage{amsmath}
\usepackage{mathrsfs}
\usepackage{verbatim}
\usepackage{enumerate}

\newcommand{\R}{\mathbb{R}}
\usepackage{cite}
\usepackage{hyperref}
\usepackage{changes}
\usepackage{color}

\numberwithin{equation}{section}

\newtheorem{theorem}{Theorem}[section]
\newtheorem{theorem*}{Theorem}
\newtheorem{remark}[theorem]{Remark}
\newtheorem{lemma}[theorem]{Lemma}

\newtheorem{proposition}[theorem]{Proposition}
\newtheorem{corollary}[theorem]{Corollary}

\newtheorem*{question*}{Question}

\newtheorem{example}[theorem]{Example}

\newtheorem{definition}[theorem]{Definition}

\newcommand\var{\mathop{\rm var}}

\newcommand\vol{{\rm vol}}

\newcommand\diam{{\rm diam}}

\begin{document}

\title{
{Nonpositive curvature,  the variance functional, and the Wasserstein barycenter}
\footnote{Y.-H.K. is supported in part by 
Natural Sciences and Engineering
Research Council of Canada (NSERC) Discovery Grants 371642-09 and 2014-0544,  as well as Alfred P. Sloan research fellowship 2012-2016.  B.P. is pleased to acknowledge the support of a University of Alberta start-up grant and National Sciences and Engineering Research Council of Canada Discovery Grant number 412779-2012. 
}}
\author{Young-Heon Kim\footnote{Department of Mathematics, University of British Columbia, Vancouver BC Canada V6T 1Z2 yhkim@math.ubc.ca} and Brendan Pass\footnote{Department of Mathematical and Statistical Sciences, 632 CAB, University of Alberta, Edmonton, Alberta, Canada, T6G 2G1 pass@ualberta.ca.}}
\maketitle
\begin{abstract} 
This paper connects nonpositive sectional curvature of a Riemannian manifold with the displacement convexity of the variance functional on the space $P(M)$ of probability measures over $M$.  We show that $M$ has nonpositive sectional curvature and has trivial topology (i.e, is homeomorphic to $\R^n$) if and only if the variance functional on $P(M)$ is displacement convex.
This is followed by a Jensen type inequality  for the variance functional with respect to Wasserstein barycenters, as well as by a  result comparing the variance of the Wasserstein and linear barycenters of a probability measure on $P(M)$ (that is, an element of $P(P(M))$). These results are applied to invariant measures under isometry group actions, giving a comparison for the variance functional between the Wasserstein projection and the $L^2$ projection to the set of invariant measures.
\end{abstract}

\section{Introduction}
In this paper, we study the influence of nonpositive sectional curvature of a complete Riemannian manifold $M$ on the geometry of the space $P(M)$ of probability measures,  equipped with the Wasserstein metric.
 
Given a probability  measure $\mu$ on $M$, the variance of $\mu$
is defined by

\begin{equation*}
\var(\mu): =\inf_{y \in M}\int_Md^2(x,y)d\mu(x),
\end{equation*}
where $d$ denotes the Riemannian distance.  A minimizer $y\in M$ in the above is often called a \emph{barycenter} of $\mu$. We are interested in the way that the variance, viewed as a functional on the space $P(M)$ of Borel probability measures on $M$, interacts with the geometry on $P(M)$  induced by the Wasserstein distance; the Wasserstein distance between $\mu,\nu \in M$ is given by

\begin{equation}\label{wassdef}
 W_2(\mu, \nu) := \inf_{\pi^1_\#\gamma =\mu, \pi^2_\#\gamma =\nu} \int_{M \times M}d^2(x,y)d\gamma(x,y),
\end{equation}
where, for $i=1,2$, $\pi^i_\#\gamma $ denotes the pushforward of $\gamma$ by the canonical projections, $\pi^1(x,y) =x$,  $\pi^2(x,y) =y$, respectively.  Recall that in general, the pushforward  $T_\#\sigma$ of a measure $\sigma$ by a map $T:X \rightarrow Y$, is defined by $T_\#\sigma(A): =\sigma( T^{-1}(A)) $ for all measurable sets $A \subset Y$.

We will show that the combination of  nonpositive sectional curvature together with  trivial topology,  is characterized by displacement convexity of the variance; that is, convexity along geodesics on $P(M)$  induced by the Wasserstein metric (see Theorem \ref{vardispcon} below).  The notion of displacement interpolation, initiated by McCann \cite{m}, gives a natural geometric  way to interpolate between two probability measures. 
 In turn, convexity of certain functionals with respect to this interpolation, known as displacement convexity, has proven to be a remarkably powerful tool in proving geometric and functional inequalities, and has found applications in physics and economics as well; see, e.g. \cite{V, V2}.

 Let us note that there are already many known characterizations of nonpositive sectional curvature; in fact there is one  involving the variance functional, due to Sturm \cite[Theorem 4.9]{sturm}, which applies to more general spaces than we consider here.  
We believe, however, that it is interesting to have a characterization involving displacement convexity, particularly in light of the now well known characterization of Ricci curvature bounds involving displacement convexity of the entropy functional, developed by many authors, including Cordero-Erausquin-McCann-Schmuckenschlager\cite{c-ems}, Otto-Villani\cite{ottovillani} and Sturm-Von-Renesse\cite{sturmvonrenesse}, and culminating in the recent work of Lott-Villani \cite{lottvillani} and Sturm \cite{sturm06,sturm06a}

Note that, unlike many other interesting displacement convex functionals, the variance functional is well defined and finite as soon as the measure has finite second moment (ie, one does not require absolute continuity with respect to volume), and is weak-* continuous.  This property makes it particularly well suited for studying sectional curvature bounds.  Heuristically, displacement interpolation moves a measures along a family of non-intersecting geodesics with fixed endpoints.   Nonnegative Ricci curvature tends to pull those geodesics apart at intermediate times; this is quantified by the displacement convexity of the entropy in the works cited above.  Our setting is slightly different; we expect nonpositive sectional curvature to contract geodesics at intermediate times in a certain sense. However, as sectional curvature is a property of two dimensional sections of the tangent space, this contraction may not be detectable by functionals which are finite only on absolutely continuous measures.  For instance, if the sectional curvature of some section is positive, but the Ricci curvature is everywhere negative, the volume of a small ball will get contracted.   However, a set which is concentrated and interpolated along the directions with positive sectional curvature can get spread out in a certain sense; the variance turns out to be an appropriate way to quantify this.

We go on to extend the convexity of the variance to convexity with respect to Wasserstein barycenters: see Theorem~\ref{thm: reduction at barycenter}.   Analagously to the definition of barycenters of measures on $M$, a barycenter $BC^W(\Omega)$ of a measure $\Omega$ on $P(M)$, which we call {\em a Wasserstein barycenter, or simply, $W_2$-barycenter, of $\Omega$} is defined as a minimizer of 

\begin{equation}\label{bcdef}
\nu \mapsto \int_{P(M)} W^2_2(\mu,\nu)d\Omega(\mu).
\end{equation}
The notion of Wasserstein barycenters  was considered by Agueh-Carlier \cite{ac} when $M$ is a subset in the Euclidean space $M\subseteq \mathbb{R}^n$ and $\Omega$ is a discrete measure on $P(M)$, and later  by the present authors \cite{KP2} for Riemannian manifolds $M$ and general probability measures $\Omega$ on $P(M)$. It extends displacement interpolation, allowing one to interpolate between several (or, in our formulation, even infinitely many) probability measures in a canonical way.  Agueh and Carlier \cite{ac} also considered convexity over  Wasserstein barycenters, as a generalization of  displacement convexity. This notion can be interpreted as an analogue of Jensen's inequality;  this  point of view was  investigated in \cite{KP2} where geometric versions of Jensen's inequality were established for displacement convex functionals on  Wasserstein spaces over Riemannian manifolds, extending the Euclidean results of \cite{ac}.

The displacement convexity of the variance should be contrasted with its behaviour with respect to  linear interpolation of measures.  When measures are interpolated linearly, it is easy to see that the variance is concave, regardless of the curvature of $M$.  Combined with the ordinary Jensen's inequality and our displacement convexity result, this implies that the variance of the Wasserstein barycenter of any measure $\Omega$ on $P(M)$ is less than or equal to the variance of its linear barycenter, if $M$ is nonpositively curved simply connected space; see Corollary~\ref{coro: var comparison}. Although this statement is not explicitly linked to convexity and concavity, we are not aware of another proof  which does not use convexity over Wasserstein barycenters.
We present a counterexample demonstrating that this inequality can fail when the curvature conditions are relaxed.

We then turn our attention to the special case when the measure $\Omega$ is induced by a left invariant measure on an  isometry group $G$  acting on $M$, and relate our work to  
the $W_2$ projection $P^W_G(\mu)$ of $\mu \in P(M)$ to the set of $G$-invariant measures on $M$. Connections between optimal transport problems and measures which are invariant under certain operations have recently begun to attract considerable attention; see \cite{momeni14}\cite{Zaev}\cite{GhM}\cite{GG}\cite{GhMa}, although  these works are primarily concerned with finding  Kantorovich solutions of the optimal transport problem with certain symmetry constraints, rather than looking at Wasserstein projections.
Our work here implies a comparison result  for the variance functional between the $L^2$ projection and the $W_2$ projection to the $G$ -invariant set. Namely, when $M$ is nonpositively curved and simply connected,
 we get, under suitable conditions on $\mu$,
 \begin{equation*}\label{varineq}
\var(P^{W}_G(\mu)) \leq \var(\mu);
\end{equation*} 
see  Corollary~\ref{coro: variance reduction projection}. 
Note that, at first glance, this inequality has no obvious connection to the barycenter of a family of measures, but we are not  aware of another simple proof of it.  
Furthermore, it is interesting when contrasted with the inequality
\begin{equation*}
\var(P^{L^2}_G(\mu)) \geq \var(\mu),
\end{equation*} 
for the $L^2$ projection $P^{L^2}_G (\mu)$ of $\mu$ onto the $G$-invariant set; see \eqref{eq: comparison}.

The paper is organized as follows: In Section 2 we establish the equivalence, on complete Riemannian manifolds, between nonpositive sectional curvature, together with simple connectedness, and displacement convexity of the variance.  In Section 3, we show that this displacement convexity extends to convexity over Wasserstein barycenters.   Section 4 is devoted to the comparison of the behaviour of the variance functional  between  linear and Wasserstein barycenters.  Finally, in Section 5, these results are applied to isometry group actions, yielding comparison results for the $L^2$ the $W_2$ projections to the set of invariant measures.

\section{Displacement convexity of the variance and nonpositive sectional curvature}
Before stating the main theorem of this section, we develop some notation.  A well known result of Brenier \cite{bren} and McCann \cite{m3} asserts that if the measure $\mu$ is absolutely continuous with respect to volume  and both $\mu$ and $\nu$ have finite variance, then there  exists a unique minimizer $\gamma$ to the minimization problem \eqref{wassdef}, and furthermore, $\gamma = (Id,F)_\#\mu$, where $F:M \rightarrow M$, is the unique mapping such that $F_\#\mu=\nu$ taking the form $F(x) =\exp_x(-\nabla \phi(x))$, where $\phi:M \rightarrow \mathbb{R}$ is a $\frac{d^2}{2}$-convex function;
 that is, $\phi$ takes the form

$$
\phi(x) =\sup_{y \in M}-\frac{d^2(x,y)}{2}-\phi^c(y)
$$  
for some $\phi^c:M \rightarrow \mathbb{R}$.  The \emph{displacement interpolant} between $\mu$ and $\nu$ is then the map $[0,1] \rightarrow P(M)$ given by $\mu_t=((1-t){\rm Id} + tDu(x))\#\mu$. 
 We note that this notion of displacement interpolation can be extended to non-absolutely continuous measures in $P(M)$; for precise definitions, we refer the reader to the books \cite{V2,ags}.  A functional $\mathcal{F}:P(M) \rightarrow \mathbb{R} \cup \infty$ is called \emph{displacement convex} if the function $t \mapsto \mathcal{F}(\mu_t)$ is convex for every displacement interpolant $\mu_t$.

This section is then devoted to the proof of the following result:
\begin{theorem}\label{vardispcon}
Assume $M$ is simply connected.  Then $M$ has nonpositive sectional curvature if and only if the variance functional is displacement convex.
\end{theorem}
\begin{proof}
 This follows from Theorem~\ref{vardison1} and Corollary~\ref{coro: convexity var} below.
\end{proof}

\subsection{Displacement convexity of variance: necessary condition}
In this subsection, a standard argument shows that if the variance is displacement convex, then the underlying Riemannian manifolds has to be simply connected and nonpositively curved. 

\begin{theorem}\label{vardison1}
 Let $M$ be a  complete Riemannian manifold. Suppose that variance is displacement convex, i.e. 
 $\var (\mu_t) \le (1-t) \var (\mu_0) + t \var (\mu_1)$ for each displacement interpolation $\mu_t$ of probability measures on $M$. 
 Then, $M$ is simply connected and has nonpositive sectional curvature $K\le 0$. 
\end{theorem}
\begin{proof}

 We first tackle the simple connectedness.  The proof is by contradiction; assume  $M$ is not simply connected.  We claim that  this implies that each point $x$ has a nonempty cut locus.  To see this, note that there are homotopically nontrivial loops from $x$ to itself.  Taking an arc-length minimizing sequence of such loops, and noting that each loop in the sequence remains in a compact subset of $M$, we can pass to a convergent subsequence and obtain a geodesic loop from $x$ to itself.  A cut locus point clearly exists along such a loop.  

By \cite[Proposition 2.5]{c-ems}, then, for any $x \in M$, there exists $y \in M$, and a small $v \in T_xM$ such that 
\begin{equation}\label{cutpt}
d^2(\exp_xv,y) +d^2(\exp_x(-v),y)-2d^2(x,y)<0.
\end{equation}
Now, take two measures $\mu_0 =\frac{1}{2}[\delta_y +\delta_ {\exp_xv}]$ and $\mu_1 =\frac{1}{2}[\delta_y +\delta_ {\exp_x(-v)}]$.  The displacement interpolant at $t=\frac{1}{2}$ is clearly $\mu_{1/2} =\frac{1}{2}[\delta_y +\delta_ {x}]$.  Note that the variances of  the doubly supported measures $\mu_0$,  $\mu_1$ and $\mu_{1/2}$ are, respectively, $\frac{1}{2}d^2(\exp_xv,y)$ +$\frac{1}{2}d^2(\exp_x-v,y)$ and $\frac{1}{2}d^2(x,y)$. This contradicts the displacement convexity of the variance.  We note that one could also use this to construct an example with  absolutely  continuous $\mu_0$ and $\mu_1$; observe that weak-* density of absolutely continuous probability measures, the weak-* continuity of the variance functional and stability of the displacement interpolation (these latter two facts are straightforward to prove; see Lemmas \eqref{weak cont} and \eqref{bc convergence} in the next section), combined with inequality \eqref{cutpt}, we can find absolutely continuous measures $\mu_0$ and $\mu_1$, whose displacement interpolant $\mu_{1/2}$ satisfies
$$
\var(\mu_1) +\var(\mu_0) < 2\var(\mu_{1/2}).
$$
This again violates the displacement convexity of the variance, yielding the desired contradiction and therefore establishing the simple connectedness of $M$.

We now turn to the sectional curvature assertion.  The proof is again by contradiction;  assume     a section $\Sigma$ of a tangent space $T_xM$  has positive sectional curvature.  Then, we can find, for  some small $\epsilon>0$,  points $x_0,x_1,y_0,y_1$ with the following properties:

\begin{eqnarray*}
d(x_0,x_1)& =&d(y_0,y_1) :=\epsilon\\
d(\gamma_0(t),\gamma_1(t)) &>&\epsilon \text{ for some $t \in (0,1)$}\\
d^2(x_0, y_0) +d^2(x_1, y_1) &\leq& d^2(x_0, y_1) +d^2(x_1, y_0) 
\end{eqnarray*}
Here $\gamma_0(t)$ and $\gamma_1(t)$ are geodesics from $x_0$ to $y_0$ and $x_1$ to $y_1$, respectively.  Now, consider optimal transport between the two measures $\mu_0=\frac{1}{2}[\delta_{x_0} +\delta_{x_1}]$ and $\mu_1=\frac{1}{2}[\delta_{y_0} +\delta_{y_1}]$; the optimal plan clearly pairs $x_0$ with $y_0$ and $x_1$ with $y_1$, and so the displacement interpolant at $t$ is  $\mu_{t}=\frac{1}{2}[\delta_{\gamma_{0}(t)} +\delta_{\gamma_{1}(t)}]$.  Therefore, the variances of $\mu_0,\mu_1$ and $\mu_{1/2}$ are, respectively, $\frac{1}{4}d^2(x_0,x_1)$, $\frac{1}{4}d^2(y_0,y_1)$ and $\frac{1}{4}d^2(\gamma_{0}(t),\gamma_{1}(t))$, and so

\begin{eqnarray*}
\var(\mu_{t}) =\frac{d^2(\gamma_0(t),\gamma_1(t))}{4}&>&\frac{\epsilon^2}{4}\\
&=&(1-t)\frac{d^2(x_0,x_1)}{4}+t\frac{d^2(y_0,y_1)}{4}\\
&=&(1-t)\var(\mu_0)+t\var(\mu_1).
\end{eqnarray*}
This contradicts displacement convexity; if one wants to consider absolutely continuous measures instead, then,  one can argue as before using an approximation argument.   
This contradiction yields the desired result.
\end{proof}

\subsection{Displacement convexity of variance: sufficient condition}

The goal of this subsection is to show that variance, as a functional on the space of probability measures, is displacement convex if the underlying  domain or manifold $M$ is simply connected and has nonpositive curvature;  if our domain is not complete, then we further assume  that it is geodesically convex. Here, by geodesic convexity of $M$, we mean that for any given two points in $M$, any minimizing geodesics connecting  these two points remains in $M$.

In fact, we prove convexity along a slightly more general family of paths than displacement interpolations or equivalently, $W_2$-geodesics  in $P(M)$.  All results in this section are obtained  using standard calculations and results in Riemannian geometry. 

\begin{definition}[$W_2$-quasi-geodesic]
Let $V$ be a measurable vector field defined a.e. on a Riemannian manifold $M$. 
Define for $t \in [0,1]$, a measurable mapping $T_t $ as $T_t (x) = \exp_x ( tV (x))$ for a.e. $x$. 
Then, for each absolutely continuous probability measure $\mu$ on $M$, we call the $1$-parameter family  $\mu_t = T_{t\#}\mu$ a {\em  $W_2$-quasi-geodesic}.
\end{definition}
Notice when $V$ is given by $\nabla \phi$ for some $c$-convex function $\phi$, (see \cite{m3}), $W_2$ quasi-geodesics become $W_2$ geodesics.
It is convenient at this point to observe a simple technical fact, which we won't need in this section but will be used in the proof of Theorem \ref{thm: reduction at barycenter} in the following section.  
\begin{lemma}[A first variation along $W_2$ quasi-geodesic]\label{lem: first variation}
Let $\mu$ be an absolutely continuous probability measure on $M$ and let $\mu_t=T_{t\#} \mu$ be a $W_2$ quasi-geodsic, given by the measurable vector field $V$, $T_t (x) = \exp_x t V(x) $ for a.e. $x$. 
Let $\gamma(t)$ be a differentiable curve  in $M$ such that $\gamma(0)$ is a barycenter of $\mu$.  Then
\begin{align*}
 \frac{d}{dt}\Big|_{t=0} W_2^2 (\delta_{\gamma(t)}, \mu(t)) 
 = -2 \int_M  \langle \exp^{-1}_x \left(\gamma(0)\right), V(x) \rangle d\mu(x) 
\end{align*}
where $\langle \cdot, \cdot \rangle$ is the Riemannian inner product. Notice that $\exp_x^{-1}(\gamma(0))$ is defined whenever $x$ is not in the cut-locus of $\gamma(0)$,  which is almost every $x$, and thus $\mu$-a.e. for the absolutely continuous measure $\mu$. 
\end{lemma}
\begin{proof}
First note that
 \begin{equation}\label{dist squared derivative}
\nabla_x  \frac{1}{2}d^2 (w, x) = -\exp_x^{-1} (w) 
\end{equation}
for a.e. $x$.  We then have

\begin{align*}
  \frac{d}{dt}\Big|_{t=0} W_2^2 (\delta_{\gamma(t)}, \mu_t ) 
 &= \int_M \frac{d}{dt}\Big|_{t=0}  d^2 ( \gamma(t), T_t (x)) d\mu (x)\\
 &= \int_M \langle \nabla_w\Big|_{w=\gamma(0)}  d^2 ( w, T_0 (x)), \gamma'(0)\rangle d \mu (x) \\
 &+  \int_M  \langle \nabla_y\Big|_{y= T_0 (x)}  d^2 ( \gamma(0), y),T_0'(x)\rangle d \mu (x)    \\
&= \Big \langle \int_M \nabla_w\Big|_{w=\gamma(0)}  d^2 ( w, T_0 (x))  d \mu (x), \gamma'(0)\Big\rangle  \\
 &+ \int_M   \langle \nabla_y\Big|_{y= T_0 (x)}  d^2 ( \gamma(0), y),T_0'(x)\rangle  d \mu (x)    
\end{align*}
 Note that in the calculations both above and below, we use the absolute continuity of the measure $\mu$ so that the non differentiability points of the distance squared do not effect the calculations. 
Now, 
$$
\int_M \nabla_w\Big|_{w=\gamma(0)}  d^2 ( w, T_0 (x))  d \mu (x) = \nabla_w\Big|_{w=\gamma(0)}\int_M  d^2 ( w, x)  d \mu (x) 
$$
vanishes because $\gamma(0)$ is a barycenter of $\mu$.  The result now follows from \eqref{dist squared derivative} and the observation that $T_0'(x) =V(x)$. 
\end{proof}

Note that convexity along $W_2$ quasi-geodeiscs implies displacement convexity.  Below we will show that the variance functional $\mu \mapsto \var(\mu)$ is convex along $W_2$ quasi-geodesics.

We will need a  simple consequence of the second variation formula of arc-length; the following Lemma is a special case of a result of Sturm \cite[Corollary 2.5]{sturm} and so we omit the proof.
\begin{lemma}[Convexity of distance squared for points along two geodesics]\label{lem: convexity of distance square}
Let $M$ be a simply connected manifold with nonpositive sectional curvature $K\le 0$.
Let $z, w: [0,1] \to M$ be two  geodesics.
Then
$t  \mapsto d^2(z(t), w(t))$ is convex.
\end{lemma}
\begin{remark}\label{rem: geodesic convex}
 This results can be easily extended to the case when $M$ is a {\em geodesically convex} domain in a complete Riemannian manifold with nonpositive curvature.
\end{remark}
  
Now, we prove the main theorem of this section:

\begin{theorem}[Convexity  along $W_2$ quasi-geodesics]\label{thm: convex W 2}
 Let $M$ be  a geodesically convex domain in  a complete simply connected manifold with nonpositive sectional curvature $K\le 0$.
 Let $\mu_t$, $a\le t\le b $,  be a $W_2$ quasi-geodesic  in $P(M)$. Let $t \in [a,b] \to w(t) \in M$ be a geodesic. 
 Then, $W_2^2(\delta_{w(t)}, \mu_t) $ is convex in $t$. 
\end{theorem}
\begin{proof}
 This is an easy corollary  of Lemma~\ref{lem: convexity of distance square} and Remark~\ref{rem: geodesic convex}. The details follow.
Note that $\mu_t = T_{t\#}\mu$ where   for a.e. $x$, $T_t(x) = \exp_x t V (x)$ for some vector field $V$ on $M$. 

We now observe
\begin{align*}
W_2^2 (\delta_{w(t)}, \mu_t) =  \int_M d^2 (w(t), z) d\mu_t (z) = \int_M d^2 (w(t), T_t (x)) d\mu_0 (x).
\end{align*}
where the first equality is from the definition of $W_2$ distance and the second equality is from 
 $\mu_t = T_{t\#} \mu_0$.
 Therefore,
\begin{align*}
 \frac{d^2}{dt^2} W_2^2(\delta_{w(t)}, \mu_t) = \int_M \frac{d^2}{dt^2} d^2 (w(t), T_t (x))  d\mu_0(x).
\end{align*}
Now, note that for a fixed $x$, $t \in [a,b] \to T_t(x)$ is a geodesic, and so using Lemma~\ref{lem: convexity of distance square}  and Remark~\ref{rem: geodesic convex}, we see 
$\frac{d^2}{dt^2} d^2 (w(t), T_t (x)) \geq 0$.  Thus,
\begin{align*}
  \frac{d^2}{dt^2} W_2^2(\delta_{w(t)}, \mu_t)  \ge 0.
\end{align*}
This completes the proof. 
\end{proof}

From Theorem~\ref{thm: convex W 2}, convexity of the variance follows immediately.  The following corollary, together with Theorem \ref{vardison1} establishes Theorem \ref{vardispcon}.
\begin{corollary}[Convexity of variance along $W_2$ quasi-geodesics]\label{coro: convexity var}
 Adopt the notation and assumptions of Theorem~\ref{thm: convex W 2}.
 Then  $\var (\mu_t)$ is convex in $t$. 
\end{corollary}
\begin{proof}
  For each interval $[\alpha, \beta] \subset [a,b]$,  choose a geodesic $ t \in [\alpha, \beta] \to M$ with $w(\alpha)$, $w(\beta)$ being the barycenter points of $\mu_\alpha$, $\mu_\beta$, respectively. Apply Theorem~\ref{thm: convex W 2} to get convexity of $W_2^2 (\delta_{w(t)}, \mu_t)$,  and note that
 $\var (\mu_t) \le W_2^2 (\delta_{w(t)}, \mu_t)$, with equality at $\alpha$ and $\beta$.
 This establishes the convexity of $\var (\mu_t)$ in $t$. 
\end{proof}
\section{Convexity of the variance functional with respect to $W_2$ barycenters.}
In this section, we use convexity along $W_2$ quasi-geodesics to prove a convexity result with respect to $W_2$ barycenters (see Theorem \ref{thm: reduction at barycenter} below); recall that $W_2$  barycenters were defined in \eqref{bcdef}.  Note that Theorem \ref{thm: reduction at barycenter} requires no regularity (ie, absolute continuity) of the measures $\mu$ in $spt(\Omega)$.  Under suitable regularity conditions on $\Omega$ (see, for example, case 1 of the proof), the argument is a  straightforward variant of the proof of a similar result (for different displacement convex functionals) in \cite{KP2}.  Much of the work in this section is related to the extension to singular measures (in which case the barycenter itself can be singular and non unique).

Throughout this section, we will assume that $M$ is a compact domain in a Riemannian manifold.  Existence of a $W_2$ barycenter of a probability measure $\Omega$ on $P(M)$ is easy to show.  The $W_2$ barycenter is not generally unique; that is,  there may be multiple minimizers in \eqref{bcdef}.  However, uniqueness is known under a mild structural condition on $\Omega$:
\begin{proposition}\label{uniquebc}
Assume $M$ is compact (e.g. a compact domain in a manifold) and that $\Omega(P_{ac}(M)) > 0$.  Then there exists a unique $W_2$ barycenter.
\end{proposition}
The proof can be found in \cite{P5} and \cite{KP2}.  We will also need the following stability result.

\begin{lemma}\label{bc convergence}
Assume $M$ is compact and suppose the probability measures $\Omega^N$ on $P(M)$ converge in the weak-* topology, to $\Omega$ (with respect to the $W_2$-distance on $P(M)$).  Then the limit of any weakly-* convergent subsequence $\bar \mu ^N$ of barycenters  of the $\Omega^N$ is a  barycenter of  $\Omega$.

\end{lemma}
\begin{proof}
 The proof is a standard argument.
Choose a weakly convergent subsequence, $\bar \mu^N \rightarrow \bar \mu$.   For any $\mu \in P(M)$, we have

\begin{eqnarray}
W_2^2(\mu, \bar \mu)& \leq& \big(W_2(\mu,\bar \mu^N) +W_2(\bar \mu^N, \bar \mu)\big)^2\\
&=&W_2^2(\mu, \bar \mu^N) +W_2^2(\bar \mu^N, \bar \mu) +2W_2(\mu, \bar \mu^N) W_2(\bar \mu^N, \bar \mu)\nonumber
\end{eqnarray}
Integrating against $\Omega^N$, we have

\begin{eqnarray}
\int_{P(M)}W_2^2(\mu, \bar \mu) d\Omega^N(\mu) & \leq\int_{P(M)}W_2^2(\mu, \bar \mu^N) d\Omega^N(\mu)+W_2^2(\bar \mu^N, \bar \mu)\\ 
&+2W_2(\bar \mu^N, \bar \mu)\int_{P(M)}W_2(\mu, \bar \mu^N) d\Omega^N(\mu) \nonumber
\end{eqnarray}
Therefore, for any $\nu \in P(M)$, we have, by definition of the barycenter $\bar \mu^N$,

\begin{align}\label{approxbc}
\int_{P(M)}W_2^2(\mu, \bar \mu) d\Omega^N(\mu)& \le\int_{P(M)} W_2^2(\mu, \nu) d\Omega^N(\mu) +W_2^2(\bar \mu^N, \bar \mu)\\
& +2W_2(\bar \mu^N, \bar \mu) \int_{P(M)}W_2(\mu, \bar \mu^N)  d\Omega^N(\mu)\nonumber
\end{align}

Now,  as  weak-* convergence is equivalent to Wasserstein convergence,   $W_2(\bar \mu^N, \bar \mu)$ tends to zero as  $N \rightarrow \infty$, and as the term $W_2(\mu, \bar \mu^N)$ is uniformly bounded by the compactness of $M$,  the last two terms on the right hand side go to zero.

By weak convergence of the $\Omega^N$, and continuity of $\mu \mapsto W_2^2(\mu, \bar \mu) $ and $\mu \mapsto W_2^2(\mu,\nu) $, the above inequality tends to 
\begin{equation}
\int_{P(M)}W_2^2(\mu, \bar \mu) d\Omega(\mu)\leq\int_{P(M)} W_2^2(\mu, \nu) d\Omega(\mu)
\end{equation}
As $\nu$ is arbitrary, this completes the proof.
\end{proof}
 Another standard argument shows:
\begin{lemma}\label{weak cont}
Assume $M$ is compact.  The mapping $\mu \mapsto \var(\mu)$ is  continuous on $P(M)$ with respect to the weak-* topology.
\end{lemma}
\begin{proof}
Suppose $\mu^N \rightarrow \mu$ in  the weak-* topology.  It is easy (in fact, almost identical to the proof in the preceding Lemma) to show that any convergent subsequence $x^N$ of barycenters of the $\mu^N$ converges to a barycenter $x$ of $\mu$.  We then need to show $\int_M d^2(y,x^N)d\mu^N(y) \rightarrow \int_M d^2(y,x)d\mu(y)$.  We have

\begin{align*}
& |\int_M d^2(y,x^N)d\mu^N(y) - \int_M d^2(y,x)d\mu(y)|\\ &\leq |\int_M d^2(y,x^N)d\mu^N(y) -   \int_M d^2(y,x)d\mu^N(y) |\\
& \quad + | \int_M d^2(y,x)d\mu^N(y)  -   \int_M d^2(y,x)d\mu(y)|
\end{align*}
As $N \rightarrow \infty$, the second term in the right hand side above goes to zero by weak convergence.  The first term can be written as 
\begin{align*}
 &|\int_M[d(y,x^N)+d(y,x)] [d(y,x^N)- d(y,x)]d\mu^N(y) | \\
 &\leq  \int_M|[d(y,x^N)+d(y,x)] [d(y,x^N)- d(y,x)]|d\mu^N(y) \\
&\leq2\diam(M) \int_M|d(y,x^N)- d(y,x)|d\mu^N(y) |\\
&\leq2\diam(M) \int_Md(x,x^N)d\mu^N(y) \\
&=2\diam(M)d(x,x^N)
\end{align*}
The result follows.
\end{proof}

\begin{corollary}\label{approx omega}
Assume $M$ is compact.
Suppose the measures $\Omega^N$ on $P(M)$ converge weakly to $\Omega$.  Then 
\begin{equation*}
\int_{P(M)}\var(\mu)d\Omega^N(\mu) \rightarrow \int_{P(M)}\var(\mu)d\Omega(\mu)
\end{equation*}
\end{corollary}
\begin{proof}
This is an immediate consequence of the continuity of $\mu \mapsto \var(\mu)$ (Lemma  \ref{weak cont}) and the definition of weak convergence.
\end{proof}

 Before we prove the main theorem of this section, we make the following observation:
\begin{lemma}\label{nonunique bc}
Suppose $\bar \mu$ is a barycenter of  the measure $\Omega$ on $P(M)$.  
 Then $\bar \mu$ is the unique barycenter of $\frac{1}{2} \delta_{\bar \mu} +\frac{1}{2}\Omega$.
\end{lemma}
\begin{proof}
For any $\nu$, we have 

\begin{eqnarray*}
\frac{1}{2} W_2^2(\nu,\bar \mu) +\frac{1}{2}\int_{P(M)} W_2^2(\mu,\nu)d\Omega(\mu) &\geq& \frac{1}{2} W_2^2(\bar \mu,\bar \mu) +\frac{1}{2}\int_{P(M)} W_2^2(\mu,\bar \mu)d\Omega(\mu)\\
&=&\frac{1}{2}\int_{P(M)} W_2^2(\mu,\bar \mu)d\Omega(\mu),
\end{eqnarray*}
with equality if only if $\nu =\bar \mu$.
\end{proof}

Now we state and prove the main theorem of this section.
\begin{theorem}[Convexity of variance with respect to the barycenter]\label{thm: reduction at barycenter}
Assume that $M$ is a compact, geodesically convex domain in a complete nonpositively curved manifold: thus, $\mu \mapsto \var(\mu)$ is convex along $W_2$ quasi-geodeiscs. 
Let $\Omega$ be a Borel probability measure on $P(M)$.  Let $\bar \mu \in P(M)$ be a $W_2$ barycenter of $\Omega$.  
Then, we have
\begin{align*}
 \var (\bar \mu) \, \le  \int_ {P(M)} \var ( \mu ) d\Omega( \mu). 
 \end{align*}
 \end{theorem}
\begin{proof}

The proof is divided into three, successively more general cases.

\textbf{Case 1:} The measure $\Omega =\sum_{i=1}^N\lambda_i \delta_{\mu_i}$ has finite support and one of the $\mu_i$ is absolutely continuous with respect to volume.

 In this case, from the result of \cite{P9}, the barycenter $\bar \mu$ is unique and absolutely continuous with respect to the volume measure (this also holds without the curvature assumption \cite{KP2}).    We will need to set up some relevant notation. Let $T^i$ be the optimal map from $\bar \mu $ to $\mu^i$; by the Brenier-McCann theorem (see \cite{m3}), for a.e. $x$, $T^i (x) = \exp_x \nabla \phi^i(x)$ for some $d^2/2$ convex function $\phi^i$. Moreover, 
by a straightforward adaptation of a result of Agueh-Carlier \cite{ac} (see also \cite{KP2} for more general cases), we have
\begin{align}\label{eqn: zero sum}
\sum_{i=1}^N\lambda_i \nabla \phi^i (x)  = 0 \quad \hbox{ for a.e. $x \in spt(\bar \mu)$. }
\end{align}
  For each $i$, let $w^i$ be a  barycenter of $\mu^i$; that is, $W_2^2 (\delta_{w^i}, \mu^i) = \var (\mu^i)$.   Let $\bar w$ be a barycenter of $\bar \mu$.  

Let $\mu_t^i=T^i_{t\#}\bar \mu=\exp_xt\nabla \phi^i(x) _{\#}\bar \mu$ be the corresponding displacement interpolations (which are of course  $W_2$ quasi-geodesics); note that $\mu^i_0 = \bar \mu, \mu^i_1 =\mu^i$.  
Consider the geodesic $t\in [0,1]\mapsto \gamma^i (t) \in M$ with $\gamma^i (0) = \bar w$, $\gamma^i (1) = w^i$, and  the function 

$$t \mapsto \Phi (t) = \sum_{i=1}^N \lambda_i W_2^2 ( \delta_{\gamma^i (t)}, \mu^i_t).$$ 
Use \eqref{eqn: zero sum} to compute 
\begin{align}\label{eqn: zero derivative} \nonumber  
& \frac{d}{dt}\Big|_{t=0} \Phi (t) 
= \sum_{i=1}^N\lambda_i  \frac{d}{dt}\Big|_{t=0} W_2^2 (\delta_{\gamma^i (t)}, \mu_t^i) \\ \nonumber
 & =\sum_{i=1}^N  \lambda_i\big(-2 \int_M  \langle \exp_{x}^{-1} \bar w, \nabla \phi^i (x) \rangle d\bar \mu(x)\big) \quad\hbox{(from Lemma~\ref{lem: first variation})}\\ \nonumber 
 & =  -2 \int_M  \langle \exp_{x}^{-1} \bar w, \sum_{i=1}^N \lambda_i \nabla \phi^i (x)  \rangle d\bar \mu(x) \quad \\ 
 & =0 \quad \hbox{(from \eqref{eqn: zero sum})}
\end{align}
Note that $\Phi(t)$ is  convex, since from Theorem~\ref{thm: convex W 2}, $W_2^2 (\delta_{\gamma^i (t)}, \mu^i (t))$ is convex in $t$. 
Combined with \eqref{eqn: zero derivative}, convexity of $\Phi$ implies 
\begin{align*}
\Phi (0) \le \Phi (t).
\end{align*}
But, notice that 
\begin{align*}
  \Phi (0) &= \sum_{i=1}^N \lambda_i W_2^2 (\bar w, \bar \mu) = \var (\bar \mu),\\
  \Phi (1) & =\sum_{i=1}^N \lambda_i W_2^2 ( \omega^i, \mu^i )   =\int_{P(M)}\var(\mu)d\Omega(\mu) .    
\end{align*}
This establishes the result in the first case.

\textbf{Case 2:} Next, we  consider the case when  $\Omega$ has a unique barycenter.   

Noting that Wasserstein space $P(M)$ over $M$ is a Polish space, we can choose a sequence $\Omega^N=\sum_{i=1}^N\lambda_i \delta_{\mu_i}$ of  finitely supported measures on $P(M)$ converging weakly-* to $\Omega$, by \cite[Theorem 6.18]{V2}.   For each $N$, we can also choose at least one of the $\mu_i$ to be absolutely continuous with respect to volume, by weak-* density of absolutely continuous measures on $M$.   By Lemma \ref{bc convergence} and uniqueness, we know that the barycenters $\bar \mu^N$ of  $\Omega^N$ converge weakly to the barycenter $\bar \mu$ of $\Omega$.  Now, by the above

$$
\var(\bar \mu^N) \leq \int_{P(M)}\var(\mu) d\Omega^N(\mu).
$$
Now, take the limit as $N \rightarrow \infty$.  The right hand side tends to $\int_{P(M)} \var(\mu) d\Omega(\mu)$ by Corollary \ref{approx omega}.  The left hand side tends to $ \var(\bar \mu)$  by Lemma  \ref{weak cont}.
This completes the proof in the case when the barycenter is unique.

\textbf{Case 3:} Finally, we consider the general case.

Let $\bar \mu$ be a (not necessarily unique) barycenter of $\Omega$.  By case 2 and Lemma \ref{nonunique bc},  we have
$$
\var(\bar \mu) \leq \frac{1}{2}\var(\bar \mu) +\frac{1}{2} \int_{P(M)} \var(\mu) d\Omega(\mu)
$$
which easily implies the desired result.
\end{proof}
 
Note that the above theorem does not hold if the curvature assumption is removed, as the following example illustrates:
\begin{example}[Sphere]\label{ex: sphere}
Let $M$ be the $2$-dimensional  round sphere of circumference $2$, i.e. the Riemannian distance from the north to south pole is $1$. 
Then, consider the two measures $\mu_0= \delta_{n}$, $\mu_1 = \delta_{s}$, where $n$ and $s$ denote the north and south pole, respectively. Let $\Omega=\frac{1}{2} (\delta_{\mu_0}+ \delta_{\mu_1})$ on $P(M)$ and note that  $\int_{P(M)} \var (\mu) d\Omega(\mu)=\frac{1}{2}\var(\delta_{n}) +\frac{1}{2}\var(\delta_{s})=0$.  There are infinitely many $W_2$-barycenters; namely, any probability measure supported on the equator is a $W_2$-barycenter of $\Omega$. In particular, $\delta_z$ is a $W_2$-barycenter for any $z$ in the equator, which has vanishing variance. This does not violate the inequality of  Theorem~\ref{thm: reduction at barycenter}. However,  uniform measure $\bar \mu$ on the equator is also a barycenter. Then note that $\var (\bar \mu) =  \int d^2(t, n) d\mu(t)$ for the north pole $n$, and so $\var (\bar \mu) = 1/4 > 0 = \int_{P(M)} \var (\mu) d\Omega(\mu)$. 
\end{example}

We close this section by noting a consequence of Theorem~\ref{thm: reduction at barycenter} which will be relevant in the next section.
 
\begin{corollary}[Variance gets reduced at the barycenter of an orbit of isometries]\label{coro: reduction coro}
   Let $G$ be a set of isometries on  a complete simply connected manifold $M$ of nonnegative curvature. Let $\Omega$ be a probability measure on $G$. 
  Consider  a 
 probability measure $\mu$ on $M$  and assume that  there exists a large geodesic ball that contains the union of the  supports of  the measures $g_\# \mu$ for all $g \in G$.  Let $\bar \mu_\Omega$ be a $W_2$ barycenter of  the measure $(g \mapsto g_\# \mu)_{\#}\Omega$ on $P(M)$.

 Then, 
\begin{align*}
 \var ( \bar \mu_\Omega) \le \var (\mu).
\end{align*}

\end{corollary}

\begin{proof}
 The corollary immediately follows from Theorem~\ref{thm: reduction at barycenter} since $\var (g_\# \mu) =\var (\mu)$ for each isometry $g$.   Note that under the nonnegative curvature and simply connected assumption, each geodesic ball is geodesically convex. 
 \end{proof}
 \begin{remark}
Although we do not pursue it here, by assuming some decay conditions on the measure $\Omega$ on $G$, as well as considering the space $P_2(M)$ of measures with finite variance, one may extend the above results to non compact cases, in particular, to include isometry group actions on the whole Euclidean space or the hyperbolic space. 
 \end{remark}
 
\section{Comparison with linear interpolation}
 In this section, we obtain, as a corollary to Theorem~\ref{thm: reduction at barycenter}, a comparison result for the variance functional between the linear barycenters and the $W_2$ barycenters.

We first consider the linear interpolation between probability measures, $\mu_t=(1-t)\mu_0+t\mu_1$.  We then have that

$$
t \mapsto \var(\mu_t) =\min_y \int_Md^2(x,y)d\mu_t(x)
$$
is an infimum of affine functions, and hence concave. Define $BC^L(\Omega) \in P(M)$ to be the linear barycenter of the measure $\Omega$ on $P(M)$; that is, for each Borel $A \subseteq M$,

$$
BC^L(\Omega)[A]:=\int_{P(M)}\mu(A)d\Omega(\mu).
$$ 
Then, the (linear) concavity of the variance and the classical (linear) Jensen's inequality implies
\begin{align}\label{eq: BCL}
 \var(BC^L(\Omega))\geq \int_{P(M)}\var(\mu)d\Omega(\mu).
\end{align}

Note that this holds for \emph{any} Riemannian manifold $M$; the sectional curvature does not play a role.  On the other hand, we have:
\begin{corollary}\label{coro: var comparison}
 Let  $M$ be a compact geodesically convex domain in a complete manifolds of nonpositive sectional curvature, $\Omega$ be a probability measure on $P(M)$, and $\bar \mu$ be its $W_2$-barycenter. 
Then,  $$
\var(\bar \mu)\leq \var(BC^L(\Omega)).
$$
\end{corollary}
\begin{proof}
 This immediately follows from  the preceding inequality combined with Theorem \ref{thm: reduction at barycenter}.
\end{proof}
 This inequality seems quite intuitive to us.  Consider the following, naive explanation.  For simplicity,  focus on the interpolation between two measures; displacement interpolation (the two measure case of $W_2$ barycenters) moves the support of one measure to the other continuously along geodesics, so that the support of the interpolant should not be much more spread out  than the supports of the two original measures.  On the other hand, the support of the linear interpolant is the union of the supports of the two original measures, and so we expect it to be more spread out (i.e, have  higher variance) than the displacement interpolant.
However, this intuition is somewhat misleading, as it does not require any assumptions on the curvature; as the following example demonstrates, the nonpositive curvature condition in Corollary \ref{coro: var comparison} is essential.

\begin{example}[Balloon on a string]
Consider the sphere $S^2$ of circumference $1$ (the ``balloon"), so the distance between the north and south poles is $\frac{1}{2}$, with a line segment of length $1$ (the ``string") attached to the south pole.  Let $x$ be the north pole, and  set $x_0=\exp_xv$,  $x_1 =\exp_x(-v)$ for some $v \in T_xM$ with $|v| =\epsilon <\frac{1}{4}$ (that is, $x_0$ and $ x_1$ are found at the same distance from the north pole, along opposite directions.)  Let $y$ be the point on the line segment at a distance $\frac{1}{2}-\epsilon$ from the south pole.  

Now, set $\mu_0 =\frac{1}{2}[\delta_{y} + \delta_{x_0}]$ and $\mu_1 =\frac{1}{2}[\delta_{y} + \delta_{x_1}]$.  Note that the south pole is the $W_2$-barycenter of both of these measures, and they each have variance $(\frac{1}{2}-\epsilon)^2$.  It is then easy to see that the south pole is also the barycenter of the linear interpolation:
$
\mu_{1/2}^L:=\frac{1}{2}[\mu_0 +\mu_1]
$
and that $\var(\mu_{1/2}^L) = (\frac{1}{2}-\epsilon)^2$ as well.  On the other hand, the displacement interpolant is given by 
$
\mu_{1/2}^{W}:=\frac{1}{2}[\delta_{y} + \delta_{x}]
$
(recalling that $x$ is the north pole) whose variance is given by 
\begin{align*}
\var(\mu_{1/2}^{W})=\frac{d^2(x,y)}{4} =\frac{(1-\epsilon)^2}{4}
>\var(\mu_t^L).
\end{align*}

Although the metric space in the example is not a smooth manifold, it can easily be smoothed out to construct smooth examples where the preceding variance inequality holds. 
\end{example}
\section{$W_2$ Projection to the $G$-invariance set}
In this section, we consider isometry group actions on the underlying space $M$, which also induce isometry group actions on $P(M)$. We are interested in the Wasserstein $W_2$-barycenter of the orbit of the group action, in relation to  functionals on $P(M)$. Our focus in this section is on the variance functional for nonpositively curved underlying space, so that we can use the results of the preceding sections. But, the similar results hold for other examples; see Remark~\ref{rmk: action} and Example~\ref{ex: entropy}.

We begin by showing that the projection onto the invariance set conincides with the barycenter of the orbit under left Haar measure.  

\begin{proposition}[Projection to $G$-invariance set]\label{prop: projection G invariant}
Let $G$ be a group of isometries on a Riemannian manifold $M$ and $H$ be a left invariant probability measure on $G$ (here, the group $G$ has to be compact).
 For a given  probability measure $\mu \in P(M)$, assume that the barycenter $BC^W_G(\mu)$ of $\Omega_\mu = (\mu \mapsto g_\#\mu)_\#H$ is unique.  Define the $G$-invariant set $I_G = \{ \nu \in P(M) \, | \, g_\# \nu = \nu, \quad \forall g \in G\}$. 
Then, 
\begin{enumerate}
 \item $BC^W_G(\mu) \in I_G$.
\item $BC^W_G(\mu)$ is the unique $W_2$ projection of $\mu$ to $I_G$; that is  $\{BC^W_G(\mu) \}= \mathop{\rm argmin}_{\nu \in I_G} W_2^2 (\nu,\mu) $, or, using the notation in the introduction,
\begin{align*}
BC^W_G(\mu) =P^W_G(\mu).
\end{align*}

\end{enumerate}
\end{proposition}
This should be a well-known standard fact from metric geometry, but, we give its proof for completeness. 
\begin{remark}\label{rem: unique}
The uniqueness condition on the barycenter is satisfied when $\mu$ is absolutely continuous with respect to volume (as then each $g_\# \mu$ is clearly absolutely continuous as well) by Proposition \ref{uniquebc}.
\end{remark}
\begin{proof}
For simplicity of notation, we denote $\bar \mu =BC^W_G(\mu)$.  We prove the two assertions below:
\\
  1.  For each $g, g' \in G$, by the isometry property, we have
$
 W_2 (g_\# \bar \mu, gg '_\#\mu)  = W_2 (\bar \mu, g'_\#\mu)
$
Thus,
\begin{align*}
 & \int_G W_2^2 (\bar \mu, g'_\#\mu) d\Omega (g') = \int_G W_2^2(g_\# \bar \mu,  g g'_\#\mu) d\Omega (g') \\ & = \int_G W_2^2 (g_\# \bar \mu,  g g'_\#\mu) d\Omega (gg') \quad \hbox{(as $\Omega$ is left invariant )}
\end{align*}
This implies that $g_\#\bar \mu$ is a barycenter of $\Omega$; by the uniqueness assumption, this shows $\bar \mu = g_\#\bar \mu$. As this holds for each $g\in G$, we have $\bar \mu \in I_G$. 

2.  Notice that if $\nu \in I_G$, then, $ W_2 (\nu, g_\#\mu) = W_2 (\nu, \mu) $ for all $g\in G$. Thus, 
\begin{align*}
W_2^2 (\nu, \mu) & =  \int_G W_2^2 (\nu, \mu) d\Omega(g)\\
& = \int_G W_2^2 (\nu, g_\#\mu) d\Omega(g) \\
& \ge \int_G W_2^2 (\bar \mu, g_\#\mu) d\Omega(g) \quad \hbox{(from the definition of $\bar \mu$)}\\
& = \int_G W_2^2 (\bar \mu, \mu) d\Omega(g)  \quad \hbox{(since $\bar \mu \in I_G $ from 1. )}\\
& = W_2^2 (\bar \mu, \mu)
\end{align*}
This shows that $\bar \mu$ is the minimum of $\{W_2^2 (\nu,\mu)\}_{\nu \in I_G}$.  Noting that the inequality is strict unless $\nu =\bar \mu$, by the uniqueness of the barycenter, completes the proof.
\end{proof}

An interesting consequence follows:
\begin{corollary}[Variance gets reduced at the $W_2$ projection to the invariant set]\label{coro: variance reduction projection}
Under the same notation as in Proposition~\ref{prop: projection G invariant}, assume further that $M$ is   a complete, simply connected nonpositively curved manifold. For each absolutely continuous probability measure  $\mu$  with compact support,
 
$$
 \var ( P^W_G\mu) \le \var (\mu).
$$
\end{corollary}

\begin{proof}
This immediately follows from Corollary~\ref{coro: reduction coro} and Proposition~\ref{prop: projection G invariant}. 
\end{proof}

If $\mu$ is absolutely continuous, with a density $f$ in $L^2$, it is straightforward to see that the linear barycenter $BC^L_G(\mu)$ of $\Omega_\mu$   is absolutely continuous as well, and that it's density is given by
$$
\bar f(x) =\int_Gf(g^{-1}(x))dH(g).
$$
This $\bar f$ is the $L^2$ minimizer of the functional
$$
h \mapsto \int_{G}||h-f\circ g^{-1}||^2dH(g)
$$
and it also coincides  with the $L^2$ projection $P^{L^2}_G(\mu)$ of $\mu$ onto the  subspace of $L^2$ functions  which are invariant under the action of $G$. From \eqref{eq: BCL} and Theorem~\ref{coro: variance reduction projection}, we have
\begin{align}\label{eq: comparison}
 \var(P^{W}_G(\mu)) \leq \var(\mu) \leq \var(P^{L^2}_G(\mu)).
\end{align}

  In particular, if we consider the case $G=SO(n)$ acting on the Euclidean space $\R^n$ (or any rotationally symmetric nonpostivley curved metric on $\R^n$, e.g. the hyperbolic metric), projecting onto the invariant set can be interpreted as finding the best rotationally invariant approximation of $\mu$.  Finding the best approximation of a measure $\mu$, in the Wasserstein sense, by a radially symmetric measure decreases the variance.  On the other hand, finding the best approximation of a measure by a radially symmetric measure in the $L^2$ sense \textit{increases} the variance.

\begin{remark}\label{rmk: action}

It is worth noting that Corollary~\ref{coro: variance reduction projection} holds whenever the variance is replaced with any functional $F$ which is convex over $W_2$-barycenters (including the three main types discovered in \cite{m},  whose convexity over $W_2$ barycenters are obtained in \cite{ac} on the Euclidean space and in \cite{KP2} on Riemannian manifolds with nonnegative Ricci curvature, and more generally on smooth metric measure spaces satisfying the $CD(K,N)$ condition for $K\ge 0$), provided the functionals are invariant under  an isometry group $G$; that is, provided $F(A_\#\mu) =\mu$ for all $\mu$ and all $A \in G$.  We feel the variance case on Hadamard manifolds (another name for complete, simply connected nonpositively curved manifolds)  is of special interest as the opposite holds true for linear barycenters.  Linear and Wasserstein projections give two ways to canonically generate a $G$-invariant measure from a given measure; one of these decreases the variance while the other increases it. 
\end{remark}

\begin{example}\label{ex: entropy}
As an illustrative example, consider the entropy functional, $F(fd\vol) =\int_{M}f(x)\ln(f(x))d\vol(x)$ on a manifold with nonnegative Ricci curvature.  It is well known that $F$ is displacement convex; our recent work \cite{KP2} extends this result to show that $F$ is in fact convex over barycenters.  For any  absolutely continuous measure $\mu$ and any compact group of isometries $G$ on $M$, we then get

\begin{equation}\label{entropy}
 F(BC^W_G(\mu)) \leq F(\mu).
\end{equation}
In the particular case when $M=S^n$ is the round sphere and $G$ is the whole rotation group, the barycenter  $BC^W_G(\mu) =\vol/\vol(S^n)$ must be uniform measure, as this is the only probability measure on $S^n$ which is invariant under this group.  It is well known that uniform measure minimizes the entropy, so this is consistent with \eqref{entropy}.  For smaller rotation groups $G$, symmetrizing with respect to $G$ (that is, projecting onto the $G$-invariant set) reduces the entropy, by \eqref{entropy}. 
\end{example}

\bibliographystyle{plain}
\bibliography{biblio}
\end{document}